\documentclass[12pt]{amsart}
\usepackage{pdfsync}
\usepackage{graphicx}
\usepackage{amssymb}
\usepackage{amsmath}
\usepackage{amsthm}
\usepackage{enumerate}
\usepackage[left=3cm,top=3.8cm,right=3cm]{geometry} 
\usepackage{verbatim}                                   

\usepackage{xcolor}
\usepackage[pagebackref]{hyperref}
\hypersetup{
   colorlinks,
    linkcolor={red!60!black},
    citecolor={blue!60!black},
    urlcolor={blue!90!black}
}

\newcommand{\dist}{\text{dist}}
\newcommand{\e}{\varepsilon}
\DeclareMathOperator{\hdim}{dim_H}

\DeclareMathOperator{\spt}{spt}

\DeclareMathOperator{\Vol}{Vol}

\newtheorem{theorem}{Theorem}
\newtheorem{definition}[theorem]{Definition}
\newtheorem{proposition}[theorem]{Proposition}
\newtheorem{corollary}[theorem]{Corollary}

\newtheorem{lemma}[theorem]{Lemma}

\numberwithin{theorem}{section}

\newtheorem{remark}[theorem]{Remark}

\newcommand{\R}{\mathbb{R}}

\def\N{{\mathbb N}}

\begin{document}
\thispagestyle{empty}
\title{On the volumes of simplices determined by a subset of $\mathbb{R}^d$}

\author{Pablo Shmerkin}
\address{Department of Mathematics, the University of British Columbia. 1984 Mathematics Road, Vancouver BC V6T 1Z2, Canada}
\email{pshmerkin@math.ubc.ca}

\thanks{P.S. was partially supported by an NSERC Discovery Grant}

\author{Alexia Yavicoli}
\address{Department of Mathematics, the University of British Columbia. 1984 Mathematics Road, Vancouver BC V6T 1Z2, Canada}
\email{yavicoli@math.ubc.ca, alexia.yavicoli@gmail.com}

\keywords{patterns, configurations, simplices, volumes, Hausdorff dimension, slices, projections}
\subjclass[2020]{28A12, 28A78, 28A80}

\begin{abstract}
  We prove that for $1\le k<d$, if $E$ is a Borel subset of $\mathbb{R}^d$ of Hausdorff dimension strictly larger than $k$, the set of $(k+1)$-volumes determined by $k+2$ points in $E$ has positive one-dimensional Lebesgue measure. In the case $k=d-1$, we obtain an essentially sharp lower bound on the dimension of the set of tuples in $E$ generating a given volume. We also establish a finer version of the classical slicing theorem of Marstrand-Mattila in terms of dimension functions, and use it to extend our results to sets of ``dimension logarithmically larger than $k$''.
\end{abstract}

\maketitle

\section{Introduction and main results}

A fruitful and highly active area of analysis is concerned with the richness of patterns inside fractal sets. A classical example, which motivated much of the development of the area, is Falconer's distance set problem: Given a set $E \subseteq \R^d$, what can be said about the Hausdorff dimension, Lebesgue measure or interior of the set of distances between points in $E$, in terms of the Hausdorff dimension of $E$? 

A huge number of generalizations of Falconer's problem have been proposed, generally by looking at configurations spanned by $k\ge 3$ points instead of two points (some of these are briefly discussed at the end of this section). If we interpret the distance between two points as the volume of the one-simplex they generate, then a natural generalization is to consider the set of $k$-volumes of simplices generated by $k+1$ points in a set $E\subset\R^d$:
\[
  \Vol_k (E)= \big\{\Vol_k(x_1, \cdots, x_{k+1}): \ x_i \in E\big\} \subseteq \R_{\ge 0}.
\]
This problem was considered by Grafakos, Greenleaf, Iosevich and Palsson: in \cite[Theorem 3.7]{GGIP}, they show that if $E\subset\R^d$ is a Borel set with $\hdim(E)>d-1+\tfrac{1}{2d}$ if $d$ is even, and $\hdim(E)>d-1+\tfrac{1}{2(d-1)}$ if $d$ is odd, then $\mathcal{L}(\Vol_d(E))>0$, where $\mathcal{L}$ denotes one-dimensional Lebesgue measure. Since a $d-1$ plane determines a single volume of a $d$-simplex, namely $0$, it seems reasonable to conjecture that this is the sharp threshold. In \cite[Theorem 3.8] {GGIP}, the authors show that if this is the right threshold for $d=2$, then it is also the correct threshold in arbitrary dimensions.

In this paper, we directly establish a strong pinned form of this conjecture, which also holds for $k$-volumes of simplices in $\R^d$ for any $k\ge 1$:
\begin{theorem}\label{Thm:Main}
  Fix $d\in \N_{\ge 2}$ and $k\in\{1,\ldots,d-1\}$. Let $E \subseteq \R^d$ be a Borel set  with $\hdim(E) >k \geq 1$. Then there exist $x_{1},\ldots,x_{k}\in E$ such that the set
  \[
    \Vol_{k+1}^{(x_1,\ldots,x_{k})}(E) :=\left\{ \Vol_{k+1}(x_1,\ldots,x_{k+1}): x_{k+1}\in E \right\}
  \]
  has positive Lebesgue measure. Moreover, when $k\ge 2$, there exist $x_{1},\ldots,x_{k}\in E$ such that $\Vol_{k+1}^{(x_1,\ldots,x_{k})}(E)$ has nonempty interior.
\end{theorem}

The proof of Theorem \ref{Thm:Main} is a very short application of the classical Marstrand-Mattila projection and slicing theorems in geometric measure theory. Nevertheless, to our knowledge this argument had not been noticed before (although we point out that a similar idea was used to study the set of angles determined by a set in \cite{HKKMMMS}). 

When $k=d-1$, we are able to obtain a much finer result. Theorem  \ref{Thm:Main} can be recast in the following form: suppose that $\hdim(E)>d-1$ for $E\subset\R^d$. Then for each $v$ in a set $V\subset [0,\infty)$ of positive measure, there is a non-empty set $X_v\subset E^{d+1}$ such that each tuple in $X_v$ spans a simplex of volume $v$. It is natural to ask whether one can also provide a lower bound on the Hausdorff dimension of $X_v$; we show that this is indeed the case, and in fact prove an essentially sharp lower bound on $\hdim(X_v)$:
\begin{theorem} \label{Thm:MainRefined}
  Let $E\subset\R^{d}$ be a Borel set with $\hdim(E)>d-1$. Then for every $t< (d+1)\dim(E)-1$ the set
  \begin{equation} \label{eq:dim-level-sets-volum}
    \Big\{ v\in\R^{+} : \hdim\{ x\in E^{d+1}: \Vol_{d}(x) =v\} \ge t \Big\}
  \end{equation}
  has positive Lebesgue measure.
\end{theorem}
At least when $E$ has equal Hausdorff and packing dimension, the numerology in this theorem is sharp (up to the endpoint): by Proposition \ref{prop:dim-slices} below, \eqref{eq:dim-level-sets-volum} implies that  $\hdim(E^d)\ge t+1$, and when $E$ has equal Hausdorff and packing dimensions one has $\hdim(E^{d+1})=(d+1)\hdim(E)$, see e.g. \cite[Theorem 8.10]{MattilaGeometry}. We are not aware of other instances of Falconer-type problems where sharp results are known for this refined ``level-set'' formulation (we note that for \emph{random} sets, this numerology is known to hold for a large variety of configurations - see \cite{ShmerkinSuomala20}).

The proof of Theorem \ref{Thm:MainRefined} also uses the Marstrand-Mattila slicing theorem as a key tool, but the argument is more involved.

While Theorem \ref{Thm:Main} is sharp as far as the Hausdorff dimension of $E$ is concerned, it is natural to ask whether one can provide a finer classification among sets of Hausdorff dimension $k$. For example, we do not know whether $\Vol_{k+1}(E)>0$ for all Borel sets $E\subset\R^d$ of non-$\sigma$-finite $k$-dimensional Hausdorff measure. In Section \ref{sec:dimension-k} we present some partial results: we show that $\Vol_{k+1}(E)>0$ still holds if $E$ is a $k$-dimensional set which is ``large enough'' in terms of a suitable gauge function, see Corollary \ref{Coro:MainGauge}. This is a consequence of a refined dimension function version of the Marstrand-Mattila slicing theorem, which may be of independent interest, and is presented in Section \ref{sec:slicing}.

To conclude the introduction, we note that several related Falconer-type problems have been intensively studied in the literature. The articles \cite{GGIP,EHI,GIM,GIT, GD} explore the measure of the set of $k$-volumes determined by $k$ points in a set $E\subset\R^d$ together with the origin. Many works, including \cite{EHI, GGIP,PRA22,PRA23}, investigate the size of the set of non-congruent $k$-point configurations determined by $E$. All these works use harmonic-analytic techniques, and it seems like our more direct approach here does not extend to those situations.

\section{Sharp dimension thresholds: proofs of Theorems \ref{Thm:Main} and \ref{Thm:MainRefined}}

\subsection{Preliminaries} 

We begin by recalling some key definitions and facts from geometric measure theory. Fix $1\le k<d$. Let $G(d,k)$ be the Grassmanian of $k$-dimensional subspaces of $\R^d$, and let $\gamma_{d,k}$ be the unique Borel probability measure on $G(d,k)$ which is invariant under the action of the orthogonal group $\mathbb{O}_d$. See \cite[\S 3.9]{MattilaGeometry} for more details. 

We denote the Grassmanian of \emph{affine} $k$-planes in $\R^d$ by $A(d,k)$. Given a $k$-dimensional subspace $W$ of $\R^d$ and $a\in \R^d$, we let $W_a:=W+a\in A(d,k)$. Sometimes we abuse notation and identify $W_a$ with the pair $(W,a)$. The natural measure on $A(d,k)$ is given by
  \begin{equation*}
    \lambda_{d,k}(\mathcal{A}) = \int_{G(d,k)} \mathcal{H}^1\{ a\in W^{\perp}: W_a\in \mathcal{A} \} \, d\gamma_{d,k}(W).
  \end{equation*}
See \cite[\S 3.16]{MattilaGeometry} for more details.

We denote the closed $\delta$-neighbourhood of a set $E\subset\R^d$ by $E(\delta):=\{x \in \R^d : \ \dist(x,E) \leq \delta\}$. Given a Radon measure $\mu$ on $\R^d$, we defined the \emph{sliced measures} $\mu_{W,a}$ supported on the affine plane $W_a$ by
\[
\mu_{W,a}(f):= \lim_{\delta \to 0} (2 \delta)^{-m} \int_{W_a(\delta)} f \ d\mu,\quad f\in C_0(\R^d).
\]
These measures are well-defined for $\mathcal{H}^m$-almost all $a \in V$ and depend on $(W,a)$ in a Borel manner, see \cite[\S 10.1]{MattilaGeometry}.

We denote the unit sphere in $\R^d$ by $S^{d-1}$, endowed with surface measure $\sigma^{d-1}$ (which is a multiple of $(d-1)$-dimensional Hausdorff measure $\mathcal{H}^{d-1}|_S$). For every $\theta \in S^{d-1}$, we let $L_{\theta}$ be the line through the origin and $\theta$, and $P_{\theta}:\R^d \to L_{\theta}$ be the orthogonal projection onto $L_{\theta}$.  Note that $G(d,1)$ is the quotient of $S^{d-1}$ by identifying antipodal points, and $\gamma(d,1)$ is the push-forward of $\sigma^{d-1}$ under this identification.

Given $0<s<d$, the $s$-energy of a finite Borel measure $\mu$ on $\R^d$ is defined as
\begin{equation*}
  I_s(\mu):=\iint \frac{1}{|x-y|^s} \ d\mu(x) d\mu(y).
\end{equation*}

We are now able to state the measure-theoretic versions of the Marstrand-Mattila projection and slicing theorems (see \cite[Theorems 5.4 and 5.5]{MattilaFourier} and \cite[Theorem 10.7]{MattilaGeometry}, respectively, for the proofs). 

\begin{theorem} \label{thm:measure-projection}
  Let $\mu$ be a finite Borel measure on $\R^d$ such that $I_s(\mu)<\infty$. Then:
  \begin{enumerate}[(a)]
    \item If $s>1$, then $P_\theta\mu$ is absolutely continuous with an $L^2$ density for $\sigma^{d-1}$-almost every $\theta\in S^{d-1}$.
    \item If $s>2$, then $P_\theta\mu$ is absolutely continuous with a continuous density for $\sigma^{d-1}$-almost every $\theta\in S^{d-1}$.
  \end{enumerate}
\end{theorem}

\begin{theorem} \label{thm:measure-slicing}
  Fix $1\le k<s<d$. Let $\mu$ be a finite Borel measure on $\R^d$. Then, for $\gamma_{d,k}$-almost every $W\in G(d,d-k)$, 
  \begin{equation} \label{eq:decomp-sliced-measures}
    \mu = \int \mu_{W,a} d\mathcal{H}^{k}(a)
  \end{equation}
  and
  \begin{equation} \label{eq:energy-sliced-measures}
    \int_{G(d,d-k)}\int_{\R^k} I_{s-k}(\mu_{W,a}) \ d\mathcal{H}^{k}(a) \ d\gamma_{d,d-k}(W) \le C_d I_s(\mu).
  \end{equation}
  Here $C_d>0$ is a constant depending only on $d$.
\end{theorem}

We state a corollary of Theorems \ref{thm:measure-projection} and \ref{thm:measure-slicing} for sets. It is obtained by considering a Frostman measure on the set $E$ \cite[Theorem 8.8]{MattilaGeometry}.
\begin{theorem} \label{thm:slicing-projection}
  Let $E\subset\R^d$ be a Borel set.
  \begin{enumerate}[(a)]
    \item \label{it:a}  If $\hdim(E)>1$, then $\mathcal{L}(P_{\theta}(E))>0$ for $\mathcal{H}^{d-1}$-almost all $\theta\in S^{d-1}$.
    \item \label{it:b} If $\hdim(E)>2$, then $P_{\theta}(E)$ has non-empty interior for $\mathcal{H}^{d-1}$-almost all $\theta \in S^{d-1}$.
    \item \label{it:c} If $1 \leq s < \hdim(E) \leq d$, then  for $\mathcal{H}^{d-1}$- almost all $\theta \in S^{d-1}$ there is an affine hyperplane $H$ with normal $\theta$ such that 
    \[
      \hdim(E \cap H)>s-1
    \]
    (In fact, there is a positive measure family of such hyperplanes.)
  \end{enumerate}
  \end{theorem}

To finish this section, we recall two inequalities relating the dimension of a set and that of its projections and slices under a Lipschitz map. 
\begin{proposition} \label{prop:dim-slices}
  Let $E\subset \R^d$ and let $g:E\to\R^k$ be a locally Lipschitz map. Suppose 
  \begin{equation*}
    \hdim(g^{-1}(x)) \ge t \quad\text{for all }x\in g(E).
  \end{equation*}
  Then 
  \begin{equation*}
      \hdim(E) \ge t + \hdim(g(E)).
  \end{equation*}
\end{proposition}
Special cases of this statement appear in  \cite[Corollary 3.3.2]{BishopPeres} and \cite[Theorem 7.7]{MattilaGeometry}; the general case is similar and be consulted in \cite[\S 2.10.25]{Federer}. By considering charts, the statement extends easily to locally Lipschitz maps between manifolds.

\subsection{Proof of Theorem \ref{Thm:Main}}

\begin{proof}[Proof of Theorem \ref{Thm:Main}]
To begin, we recall that 
\[
  \Vol_k(x_1, \cdots, x_{k+1})=\frac{1}{k} \dist(x_{k+1}, W) \Vol_{k-1}(x_1, \cdots, x_{k}),
\] 
where $W$ is the affine $(k-1)$-plane spanned by $\{x_1, \cdots, x_{k}\}$.

Since claims \eqref{it:a} and \eqref{it:c} in Theorem \ref{thm:slicing-projection} hold simultaneously for almost all $\theta$, we can fix $\theta$ and a hyperplane $H$ normal to $\theta$ so that $\hdim (E \cap H)>k-1$ and $\mathcal{L}(P_{\theta}(E))>0$. 

Since $\hdim (E \cap H)>k-1$, there exist $y_1, \cdots, y_{k+1} \in E \cap H$ which are affinely independent (otherwise, $E\cap H$ would be contained in a $(k-1)$-plane, implying that $\hdim(E\cap H)\le k-1$). Since $\mathcal{L}(P_{\theta}(E))>0$, we get
\begin{align*}
\mathcal{L}(\Vol_{k+1}(E)) &\geq \mathcal{L} \{\Vol_{k+1}(y_1, \cdots, y_{k+1}, x_{k+2}): \ x_{k+2} \in E\} \\
&\geq \frac{\Vol_{k}( y_1, \cdots, y_{k+1})}{k+1} \cdot \mathcal{L} \{\dist(x_{k+2},H): \ x_{k+2} \in E\} \\
&=\frac{\Vol_{k}( y_1, \cdots, y_{k+1})}{k+1} \cdot \mathcal{L}(P_{\theta}(E))>0.
\end{align*}

The claim of non-empty interior when $k\ge 2$ (so that $\hdim(E)>2$) follows in the same way, using claim \eqref{it:b} of Theorem \ref{thm:slicing-projection} instead of \eqref{it:a}.
\end{proof}

\subsection{Proof of Theorem \ref{Thm:MainRefined}}

\begin{proof}[Proof of Theorem \ref{Thm:MainRefined}]
  Fix $d-1<s<\hdim(E)$. By Frostman's Lemma \cite[Theorem 8.8]{MattilaGeometry}, there is a Borel probability measure $\mu$ supported on $E$ such that $I_{s}(\mu)<+\infty$.

  By Theorem \ref{thm:measure-slicing}, for $\gamma_{d,d-1}$-almost every $H\in G(d,d-1)$ there is a family of sliced measures $\{\mu_{H,a}:a\in H^{\perp}\}$ supported on $H_a$ and depending measurably on $(H,a)$, such that \eqref{eq:decomp-sliced-measures} and \eqref{eq:energy-sliced-measures} hold.

  Next, we define a measure $\rho$ on $E^{d}$ as
  \begin{equation*}
    \rho:=\int \mu_{H,a}^{\times d} \, d\lambda_{d,d-1}(H,a) = \int_{G(d,d-1)}\int_{H^{\perp}}  \mu_{H,a}^{\times d} \,d\mathcal{H}^{1}(a) \,d\gamma_{d,d-1}(H),
  \end{equation*}
  where $\mu_{H,a}^{\times d}$ denotes the $d$-fold Cartesian power of $\mu_{H,a}$.
   
  For any $H\in G(d,d-1)$, let 
  \begin{equation} \label{eq:G-H-mu}
    \mathcal{G}_H^{\mu} = \left\{ a\in H^{\perp} : |\mu_{H,a}|>0 \text{ and } I_{s-1}( \mu_{H,a})<\infty \right\}.
  \end{equation}
 It follows from Theorem \ref{thm:measure-slicing} that $\mathcal{H}^1(\mathcal{G}_H^{\mu})>0$ for $H$ in a subset  $G^{\mu}(d,d-1)\subset G(d,d-1)$ of full $\gamma_{d,d-1}$-measure. 
  
We claim that $\Vol_{d-1}(x_1,\ldots,x_{d})>0$ for $\rho$-almost all $(x_{1},\ldots,x_{d})$. Indeed, let $H\in G^{\mu}(d,d-1)$ and $a\in \mathcal{G}_H^{\mu}$, so that $\mu_{H,a}$ is a finite Borel measure on $H_a$ with $I_{s-1}(\mu_{H,a})<\infty$. Since $s-1>d-2$, $\mu_{H,a}$ can't give positive mass to any $(d-2)$-plane. Hence, for any fixed affinely independent $x_1,\ldots, x_j\in H_a$ with $j\le d-1$, we have that $x_1,\ldots,x_j,x_{j+1}$ are affinely independent for $\mu_{H,a}$-almost all $x_{j+1}$. The claim now follows from Fubini and induction in $j$.

By the claim, the map $x\mapsto W(x)$, where $W(x)$ is the affine hyperplane determined by $x=(x_1,\ldots,x_{d})\in E^{d}$, is well-defined $\rho$-almost everywhere.

Given $x=(x_1,\ldots,x_{d})$ with $W(x)=H_a$, let
  \begin{equation} \label{eq:def-Vol-x}
    \widetilde{\Vol}_{d}(x) = \left\{\frac{\Vol_{d-1}(x)}{d} \cdot \big|b-a\big| : b\in \mathcal{G}_{H}^{\mu} \right\}.
  \end{equation}
Since (by the claim and the definition of $\rho$) the push-forward of $W(x)$ under $\rho$ is well-defined and equals $\lambda_{d,d-1}$, and since $\mathcal{H}^1(\mathcal{G}_{H}^{\mu})>0$ for $\gamma_{d,d-1}$-almost all $H$, it follows that 
\begin{equation*}
\mathcal{L}\left(\widetilde{\Vol}_{d}(x)\right)>0 \quad\text{for $\rho$-almost all $x\in E^{d}$}.
\end{equation*}
Moreover, $\widetilde{\Vol}_{d}(x)$ for $x=(x_1,\ldots,x_{d})\in E^{d}$ is a subset of the set of volumes of simplices generated by $x_{1},\ldots,x_{d}$ and a final point $x_{d+1}\in E$.
  
By Fubini's theorem, 
  \begin{equation*}
   \big(\rho\times\mathcal{L}\big)\big\{(x,v): v\in \widetilde{\Vol}_{d}(x)\big\}>0
  \end{equation*}
  and hence, by Fubini's theorem again, there is a set $V\subset [0,\infty)$ with $\mathcal{L}(V)>0$ such that for all $v\in V$ we have
  \begin{equation} \label{eq:def-Fv}
    \rho(F_v)>0, \quad\text{where }F_v=\big\{ x\in E^{d}: v\in \widetilde{\Vol}_{d}(x) \big\}.
  \end{equation}

  We claim that for any set $F$ with $\rho(F)>0$ we have
  \begin{equation} \label{eq:Fv-claim}
    \hdim(F) \ge ds.
  \end{equation}
  Indeed, the map $W(x)$ is locally Lipschitz on its domain and, as we saw before, is well-defined $\rho$-almost everywhere. Moreover, by the definition of $\rho$, the image $W(F)$ has $\gamma_{d,d-1}$-measure $>0$, and in particular full Hausdorff dimension $d$. 
  
  Since, by Theorem \ref{thm:measure-slicing}, $I_{s-1}(\mu_{H,a})<\infty$ for $\gamma_{d,d-1}$-almost all $(H,a)$, we may assume without loss of generality that $I_{s-1}(\mu_{H,a})<\infty$ for all $(H,a)\in W(F)$. It follows that, for any $H_a\in W(F)$,
  \begin{equation*}
    I_{d(s-1)}\big(\mu_{H,a}^{\times d}|_{F}\big) \le I_{d(s-1)}\big(\mu_{H,a}^{\times d}\big)<\infty,
  \end{equation*}
   and hence 
  \begin{equation*}
   \hdim\{ x\in F: W(x)=H_a \} \ge \hdim\big(F\cap H_a^{\times d}\big) \ge d(s-1).
  \end{equation*}
  Proposition \ref{prop:dim-slices} applied to $F$ and the map $W$  now yields the claimed bound \eqref{eq:Fv-claim}.

Fix $v\in V$ for the rest of the proof. Pick  $x\in F_v$ and let $W(x)=H_a$. By the definitions  \eqref{eq:def-Vol-x} and \eqref{eq:def-Fv}, there exists $b\in \mathcal{G}_H^{\mu}$ such that 
  \begin{equation*}
    v=\frac{\Vol_d(x)}{d}\cdot  |b-a|.
  \end{equation*}
  By the definition \eqref{eq:G-H-mu}, it follows that $H_a=H_b$ and $I_{s-1}(H_a)<\infty$. In particular, $E\cap H_a$ has Hausdorff dimension $\ge s-1$.
  
  We have shown that  $\Vol_{d}(x_1,\ldots,x_{d+1})=v$ for all $(x_1,\ldots,x_{d+1})$ in the set 
  \begin{equation} \label{eq:set-volume-v}
    \big\{ (x_1,\ldots,x_{d+1}): (x_1,\ldots,x_d)\in F_v, x_{d+1}\in E\cap W(x_1,\ldots,x_d)\big\}.
  \end{equation}
  Applying Proposition \ref{prop:dim-slices} to the projection of this set to the last coordinate, the claim \eqref{eq:Fv-claim} yields that the set defined in \eqref{eq:set-volume-v} has Hausdorff dimension $\ge (d+1)s-1$.
  Since $s$ is arbitrarily close to $\hdim(E)$, this completes the proof.
\end{proof}

\section{A finer slicing theorem}
\label{sec:slicing}

In this section we obtain a finer version of the Marstrand-Mattila slicing theorem \cite[Theorem 10.10]{MattilaGeometry}, in terms of gauge functions. We begin by recalling the definition of gauge functions and generalized Hausdorff measures, and then we state the theorem. 

\begin{definition}[Gauge functions]
We say that $\varphi:\R_{\geq 0}\to\R_{\geq 0}$ is a \emph{gauge function} (or dimension function) if it right-continuous, increasing,  $\varphi(0)=0$, and $\varphi(t)>0$ if $t>0$. We denote the set of all gauge functions by $\mathcal{G}$

We endow $\mathcal{G}$ with the partial order 
\[
  \varphi_2 \prec \varphi_1 \text{ if } \lim_{x \to 0^+} \frac{\varphi_1(x)}{\varphi_2(x)} =0.
\]
\end{definition}

\begin{definition}[Generalized Hausdorff measures]
Let $\varphi\in\mathcal{G}$. We define the generalized Hausdorff measure associated to $\varphi$ as
\[
  \mathcal{H}^\varphi(E):=\lim_{\delta\to 0}  \mathcal{H}^\varphi_\delta (E)\in [0,+\infty],
\]
where $\mathcal{H}^\varphi_\delta (E):=\inf \left\{ \sum_{i}\varphi(|U_i|) : \{U_i\}_{i} \text{ is a } \delta \text{-covering of } E \right\}$.
\end{definition}
It is well known and easy to see that if $\varphi_2 \prec \varphi_1$ and $\mathcal{H}^{\varphi_2}(E)>0$ for some set $E$, then $E$ has non-$\sigma$-finite $\mathcal{H}^{\varphi_1}$-measure.

\begin{definition}[Generalized energies]
Let $\varphi$ be a gauge function, and let $\mu$ be a Radon measure on $\R^d$. We define the $\varphi$-energy of $\mu$ as
\[
I_{\varphi}(\mu):=\iint \frac{1}{\varphi(|x-y|)} \ d\mu(x) d\mu(y).
\]
\end{definition}

Recall that if $\tilde{\varphi}:\R\to\R$ is a right-continuous function,  its \emph{pseudo-inverse} is defined as
\[
{\tilde{\varphi}}^{-1}(y):=\inf \{x \in \R: \tilde{\varphi}(x) \geq y\}.
\]
Because of right-continuity, we have $\tilde{\varphi} ({\tilde{\varphi}}^{-1}(y))=y$ for all $y$.

\begin{theorem}\label{Thm:HausdorffMeasures}
Fix integers $1\le m < d$. Let $\varphi, \psi$ be gauge functions such that
\begin{equation} \label{eq:assumption-gral-slicing}
  \int_{0}^{1} r^{-2} (\varphi\circ [x^m\psi]^{-1})(r)\,dr<\infty
\end{equation}
Let $E \subseteq \R^d$ be a Borel set with $\mathcal{H}^{\varphi}(E)>0$. Then, for $\gamma_{d,d-m}$-almost every $W\in G(d,d-m)$,
 \[
 \mathcal{H}^{\psi}(E \cap W_a)>0 \text{ for a set of } a \in W^{\perp} \text{ of positive } \mathcal{H}^m \text{-measure}.
 \]
\end{theorem}
A class of functions satisfying the theorem is given by $\varphi(x)= [x\cdot \log^{-a}(1/x)]^k$ and $x^m \psi(x)=[x\cdot \log^{-b}(1/x)]^{k}$, for any $k>0$, $a>1$ and $0<b<a-1$.

For the proof of this theorem, we follow the proof of the classical case as presented in \cite{MattilaGeometry}, with suitable adaptations. We begin by recalling the following lemma, which is a variant of Frostman's lemma for gauge functions. See \cite[Lemma 3.1.1]{BishopPeres} for its proof.
\begin{lemma}[Generalized Frostman's Lemma]\label{Lem:GeneralizedFrostman}
For every $d$ there is a constant $C_d>0$ such that the following holds. Let $\varphi$ be a gauge function, and let $E \subseteq \R^d$ be a Borel set with $\mathcal{H}^{\varphi}(E)>0$. Then there exists a Radon measure $\mu$ supported on $E$ such that 
\begin{equation} \label{eq:gauge-frostman}
  \mu (B(x,r))\leq \frac{C_d}{\mathcal{H}^{\varphi}(E)} \varphi(r) \text{ for all }r>0.
\end{equation}
\end{lemma}

\begin{lemma}\label{Lem:FiniteEnergy}
Let $\mu$ be a probability measure supported on $E$ satisfying \eqref{eq:gauge-frostman} for some $\varphi\in\mathcal{G}$. Let $\tilde{\varphi}$ be a right-continuous function such that 
\[
  \int_{1}^\infty \varphi (\tilde{\varphi}^{-1}(1/u))\,du<\infty.
\]
Then, $I_{\tilde{\varphi}}(\mu)<\infty$.
\end{lemma}

\begin{proof}
By Fubini,
\begin{align*}
\int \frac{1}{\tilde{\varphi}(\|x-y\|)} \ d\mu(y) &= \int_{0}^{\infty} \mu\left\{y: \ \frac{1}{\tilde{\varphi}(\|x-y\|)}\geq u \right\} \ du\\
&\leq \int_{0}^{\infty} \mu\big(B(x,\tilde{\varphi}^{-1}(1/u))\big) \ du\\
&\leq \int_{0}^{1} 1 \ du + \frac{C_d}{\mathcal{H}^{\varphi}(E)} \int_{1}^{\infty} \varphi(\tilde{\varphi}^{-1}(1/u)) \ du <\infty.
\end{align*}
\end{proof}

\begin{theorem}\label{Thm:ComparingEnergies}
Let $m<d$, and let $\tilde{\varphi}$ be a continuous gauge function such that $\psi(x):=\tilde{\varphi}(x) x^{-m}$ is also a gauge function. Let $\mu$ a Radon measure on $\R^d$. Then
\[
\iint_{W^{\perp}} I_{\psi}(\mu_{W,a}) \ d\mathcal{H}^m(a) \ d\gamma_{d,d-m}(W) \leq C_d \, I_{\tilde{\varphi}}(\mu).
\]
\end{theorem}

\begin{proof}
Using \cite[Equation (10.5)]{MattilaGeometry} applied to the lower semicontinuous function $x\mapsto\frac{1}{\psi(x-y)}$ and Fatou's Lemma,  we get
\begin{equation*}
    I_{\psi}(\mu_{W,a})\leq \liminf_{\delta \to 0} (2 \delta)^{-m} \iint_{W_a^{(\delta)}}  \frac{1}{\psi(\|x-y\|)} \ d\mu(x) d \mu_{W,a}(y) .
\end{equation*}
Using this, Fubini, and \cite[Inequality (10.6)]{MattilaGeometry} with 
\begin{equation*}
  B(x):=\big\{a \in W^{\perp}: \ x\in W_a^{(\delta)} \big\},
\end{equation*}
so that $P^{-1}_{W^{\perp}}(B(x))=\{y: \ |P_{W^{\perp}}(x-y)| \leq \delta\}$, we have:
\begin{align*}
I(W)&:=\int_{W^{\perp}} I_{\psi}(\mu_{W,a}) \ d\mathcal{H}^m(a) \\
&\leq \liminf_{\delta \to 0} (2 \delta)^{-m} \iint_{B(x)} \int \frac{1}{\psi(\|x-y\|)} \ d \mu_{W,a}(y) \ d\mathcal{H}^m(a) \ d\mu(x) \\
& \leq \liminf_{\delta \to 0} (2 \delta)^{-m} \iint_{\{y: \ |P_{W^{\perp}}(x-y)| \leq \delta\}} \frac{1}{\psi(\|x-y\|)} \ d \mu (y) \ d\mu(x).
\end{align*}

Using Fubini again, \cite[Lemma 3.11]{MattilaGeometry} and, finally, the definition of $\psi$, we conclude that
\begin{align*}
&\int_{G(d,d-m)} I(W)\ d\gamma_{d,d-m}(W) \\
& \le \liminf_{\delta \to 0} (2 \delta)^{-m} \iiint_{\{y: \ |P_{W^{\perp}}(x-y)| \leq \delta\}} \frac{1}{\psi(\|x-y\|)} \ d \mu (y) \ d\mu(x) \ d\gamma_{d, d-m}(W)\\
&=\liminf_{\delta \to 0} (2 \delta)^{-m} \iint \frac{1}{\psi(\|x-y\|)} \,\gamma_{d, d-m}(\{W: \ |P_{W^{\perp}}(x-y)| \leq \delta\}) \ d \mu (y) \ d\mu(x) \\
&\leq \liminf_{\delta \to 0}  (2 \delta)^{-m}  \iint \frac{\|x-y\|^m}{\tilde{\varphi}(\|x-y\|)} \,\cdot\, C_d \,\delta^m \,\|x-y\|^{-m} \ d\mu(y) \ d\mu(x)\\
&= 2^{-m} C_d I_{\tilde{\varphi}}(\mu).
\end{align*}
\end{proof}

\begin{lemma}\label{Lem:AntiFrostman}
Let $\nu$ be a positive finite measure on $\R^d$ with $E=\spt(\nu)$, and let $\psi$ be a gauge function such that $I_{\psi}(\nu)<\infty$.
Then, there exists $F \subseteq E$ of positive $\nu$ measure  and a constant $C$ so that
\[
  \nu|_F (B(x,r)) \leq C \psi (r) \text{ for all }x\in\R^d ,r>0.
\]
\end{lemma}

\begin{proof}
Take $C>0$ large enough so that 
\[
  F:=\left\{x: \ \int \frac{1}{\psi(\|x-y\|)} \ d\nu (y) \leq C\right\}
\]  
has positive $\nu$-measure. Then, using that a gauge function is non-decreasing,
\[
  \nu|_F (B(x,r)) =\int_{F \cap B(x,r)} \frac{\psi(\|x-y\|)}{\psi(\|x-y\|)}  d\nu (y) \leq C \,\psi(r) .
\]
\end{proof}

We can now conclude the proof of Theorem  \ref{Thm:HausdorffMeasures}.
\begin{proof}[Proof of Theorem \ref{Thm:HausdorffMeasures}]
Since by hypothesis $\mathcal{H}^{\varphi}(E)>0$, by Lemma \ref{Lem:GeneralizedFrostman}, there exists a measure $\mu$ supported on $E$ so that \[\mu (B(x,r))\leq \frac{C_d}{\mathcal{H}^{\varphi}(E)} \varphi(r) \text{ for all }r>0.\]
We may assume that $\mu$ is a probability measure.

By the assumption \eqref{eq:assumption-gral-slicing} and a change of variables,
\[
\int_{1}^\infty \varphi \circ [x^m\psi]^{-1}(1/u)  \,du<\infty.
\] 
Thus, we get from Lemma \ref{Lem:FiniteEnergy} that  $I_{x^m\psi}(\mu)<\infty$. Hence, we can apply Theorem \ref{Thm:ComparingEnergies} to get
\[
\iint_{W^{\perp}} I_{\psi}(\mu_{W,a}) \ d\mathcal{H}^m(a) \ d\gamma_{d,d-m}(W)< \infty.
\]
This implies that
\begin{equation}\label{Eq:EnergySlicedMeasure}
I_{\psi}(\mu_{W,a}) < \infty \quad\text{for }\gamma_{d,d-m}\text{-almost all } W \text{ and } \mathcal{H}^m\text{-almost all } a\in W^{\perp}.
\end{equation}

On the other hand, since $\psi\in\mathcal{G}$, we have $x^m \ge x^m \psi(x)$ if $|x|$ is sufficiently small.  Since $I_{x^m\psi}(\mu)$, it follows that also $I_{x^m}(\mu)<\infty$. Therefore, we get from \cite[Theorem 9.7]{MattilaGeometry} that $P_V(\mu) \ll \mathcal{H}^m$ for $\gamma_{d,m}$-almost every $V \in G(d,m)$. Hence, by \cite[Equation (10.6) and next line] {MattilaGeometry}, we get
\[
\int_{W^\perp} \mu_{W,a}(W_a) \ d \mathcal{H}^m(a)=\mu(\R^d)>0,
\]
for $\gamma_{d,d-m}$-almost all $W$ (note that $\gamma_{d,m}$ is the push-forward of $\gamma_{d,d-m}$ under $W\mapsto W^{\perp}$). Therefore, 
\[
|\mu_{W,a}|>0 
\] 
for $\gamma_{d,d-m}$-almost all $W$ and $a$ in a subset of $W^{\perp}$ of positive $\mathcal{H}^m$-measure. Fix such a pair $(W,a)$ for the rest of the proof.

By Lemma \ref{Lem:AntiFrostman} applied to $\mu_{W,a}$, there is a set $F\subset \spt\mu_{W,a}\subset E\cap W_a$ with $\mu_{W,a}(F)>0$ such that
\[
  \mu_{W,a}|_F (B(x,r)) \leq C\, \psi(r) \text{ for all }x,r.
\]
Then, for every covering by balls $\{B(x_i,r_i)\}_i$ of $E\cap W_a$ (and in particular of $F$) we have
\[
  0< \mu_{W,a} (F) \leq \sum_i \mu_{W,a}|_F (B(x_i,r_i)) \leq C \sum_i \psi (r_i).
\]
This shows that $\mathcal{H}^{\psi} (E \cap W_a)>0$, completing the proof.
\end{proof}

\section{Partial results in the critical dimension}
\label{sec:dimension-k}

\subsection{A sufficient condition in terms of gauge functions}

As a consequence of Theorem \ref{Thm:HausdorffMeasures}, we have a finer version of Theorem \ref{Thm:Main} for dimension functions.
\begin{corollary}\label{Coro:MainGauge}
  Fix $1\le k<d$. Let $\varphi$ be a gauge function such that there exists another gauge function $\psi$  such that $x^{k-1} \prec \psi$, and
\[
\int_{0}^{1} r^{-2}(\varphi \circ [x\psi]^{-1})(r)\,dr<\infty
\] 
In particular, this holds for $\varphi(x)= [x\cdot \log^{-a}(1/x)]^k$ for any $a>1$.

Let $E \subseteq \R^d$ ($d \geq 2$) be a Borel set with $\mathcal{H}^{\varphi}(E)>0$. Then,
\[
  \mathcal{L}\big(\Vol_{k+1}(E)\big)>0. 
\] 
\end{corollary}
\begin{proof}
As corollary of Theorem \ref{Thm:HausdorffMeasures}, we have that, for almost every $\theta \in S^{d-1}$,
\begin{enumerate}
\item[(a)] $P_{\theta}(E)$ has positive Lebesgue measure,
\item[(b)] there is an affine hyperplane $H$ with normal vector $\theta$ so that $\mathcal{H}^{\psi}(E \cap H)>0$.
\end{enumerate}
Fix $\theta$ satisfying both conclusions and a hyperplane $H$ as in (b). Since $\mathcal{H}^{\psi}(E \cap H)>0$ with $x^{k-1} \prec \psi$, the set $E\cap H$ has non-$\sigma$-finite $(k-1)$-dimensional Hausdorff measure, and hence there exist $y_1, \ldots, y_{k+1} \in E \cap H$ which are affinely independent.

Since $P_{\theta}(E)$ has positive Lebesgue measure, we conclude that
\begin{align*}
\mathcal{L}(\Vol_{k+1}(E)) &\geq \mathcal{L} \{\Vol_{k+1}(y_1, \cdots, y_{k+1}, x_{k+2}): \ x_{k+2} \in E\} \\
&\geq \frac{\Vol_{k}( y_1, \cdots, y_{k+1})}{k+1} \mathcal{L} \{\dist(x_{k+2},H): \ x_{k+2} \in E\}>0.
\end{align*}

\end{proof}

In the case $k=d-1$, we have the following extension of Theorem \ref{Thm:MainRefined}.
\begin{theorem}
  Fix $1\le k<d$. Let $\varphi$ be a gauge function such that there exists another gauge function $\psi$  such that $x^{d-2} \prec \psi$, and
  \[
  \int_{0}^{1} r^{-2}(\varphi \circ [x\psi]^{-1})(r)\,dr<\infty
  \] 
  In particular, this holds for $\varphi(x)= [x\cdot \log^{-a}(1/x)]^{d-1}$ for any $a>1$.
  
  Let $E \subseteq \R^d$ be a Borel set with $\mathcal{H}^{\varphi}(E)>0$. Then there exists a set $V\subset [0,\infty)$ with $\mathcal{L}(V)>0$ such that for all $v\in V$ we have
  \[
    \hdim\big\{ (x_1,\ldots,x_{d+1})\in E^{d+1}: \Vol_{d}(x_1,\ldots,x_{d+1})=v \big\} \ge (d+1)(d-1)-1
  \] 
\end{theorem}
This follows exactly as in the proof of Theorem \ref{Thm:MainRefined}, using Theorem \ref{Thm:ComparingEnergies} with $m=1$ in place of Theorem \ref{thm:measure-slicing}. We remark that \eqref{eq:decomp-sliced-measures} still holds in this case, since the assumption on $\widetilde{\varphi}$ in Theorem \ref{Thm:ComparingEnergies} implies that $x\prec \varphi$, which in turn implies that $P_{\theta}\mu$ is absolutely continuous for $\mathcal{H}^{d-1}$-almost every $\theta$ (see \cite[Theorem 9.7]{MattilaGeometry}), and in turn this yields \eqref{eq:decomp-sliced-measures} by \cite[p.141]{MattilaGeometry}. The details are left to the interested reader.

\subsection{Dimension of the set of areas}

If $E \subseteq \R^{2}$ with $\hdim(E)=1$, then the set of areas spanned by $E$ might be a singleton (if $E$ is contained in a line). But what if $E$ is not contained in a line? As a corollary of recent radial projection results \cite{OSW23}, we have the following result.
\begin{lemma}
  Let $E\subset\R^{2}$ be a Borel set with $\hdim(E)\le 1$ which is not contained in a line. Then
  \[
    \hdim(\Vol_2(E))\ge \hdim(E).
  \]
\end{lemma}
\begin{proof}
  By \cite[Theorem 1.1]{OSW23}, the set $D(E)\subset S^1$ of directions spanned by pairs of distinct points in $E$ has Hausdorff dimension $\ge \hdim(E)$. By Kaufman's projection theorem (see \cite[Theorem 5.1]{MattilaFourier}), for any $\e>0$ there is $\theta\in D(E)$ such that $\hdim(P_\theta(E))\ge \hdim(E)-\e$. The usual base times height argument, using points $y_1,y_2\in E$ spanning the direction $\theta$, now gives the claim
\end{proof}

Very recently, K.~Ren \cite{Ren23} generalized the radial projection theorem from \cite{OSW23} to higher dimensions. It seems likely that this allows to generalize the above lemma to higher dimensions as well. We hope to address this in a future revision.


\end{document}